\documentclass[11pt,a4paper,reqno]{amsart}
\usepackage{amsmath,amsfonts,amssymb,amsthm,amscd}
\usepackage[english]{babel}
\usepackage[mathscr]{eucal}
\usepackage{graphicx}
\usepackage{hyperref}

\renewcommand{\b}{\beta}

\renewcommand{\d}{\delta}
\newcommand{\D}{\Delta}

\newcommand{\n}{\nu}
\newcommand{\om}{\omega}
\newcommand{\OM}{\Omega}
\renewcommand{\r}{\rho}

\renewcommand{\t}{\tau}
\renewcommand{\th}{\theta}

\newcommand{\f}{\varphi}
\newcommand{\F}{\Phi}

\newcommand{\y}{\eta}

\newcommand{\p}{\psi}
\renewcommand{\P}{\Psi}

\newcommand{\C}{{\mathbb C}}

\newcommand{\N}{{\mathbb N}}
\newcommand{\R}{{\mathbb R}}

\newcommand{\Pb}{{\mathbb P}}

\newcommand{\Bc}{{\mathcal B}}
\newcommand{\Cc}{{\mathcal C}}

\newcommand{\Dc}{{\mathcal D}}

\newcommand{\Hc}{{\mathcal H}}

\newcommand{\Vc}{{\mathcal V}}

\newcommand{\curl}{{\rm curl}\,}

\newcommand{\diver}{{\rm div}\,}

\newcommand{\pd}{\partial}

\newcommand{\df}[2]{\frac{{\rm d} {#1}}{{\rm d} {#2}}}

\newcommand{\supp}{\operatorname{supp\,}}

\newcommand{\loc}{\operatorname{{loc}}}

\newtheorem{theorem}{Theorem}[section]
\newtheorem{proposition}[theorem]{Proposition}
\newtheorem{lemma}[theorem]{Lemma}

\newtheorem{definition}[theorem]{Definition}

\theoremstyle{remark}
\newtheorem{remark}[theorem]{Remark}

\numberwithin{equation}{section}

\newcommand{\na}{\nabla}

\begin{document}


\title{3D viscous incompressible fluid around one thin obstacle}
\author{C. Lacave}

\address[C. Lacave]{Universit\'e Paris-Diderot (Paris 7)\\
Institut de Math\'ematiques de Jussieu - Paris Rive Gauche\\
UMR 7586 - CNRS\\
B\^atiment Sophie Germain \\
Case 7012\\
75205 PARIS Cedex 13\\
France.} \email{lacave@math.jussieu.fr}


\begin{abstract}
In this article, we consider Leray solutions of the Navier-Stokes equations in the exterior of one obstacle in 3D and we study the asymptotic behavior of these solutions when the obstacle shrinks to a curve or to a surface. In particular, we will prove that a solid curve has no effect on the motion of a viscous fluid, so it is a removable singularity for these equations.
\end{abstract}

\maketitle

\section{Introduction}

The present article is devoted to the stability of the Navier-Stokes equations when one obstacle shrinks to a curve or a surface, and to determine the influence of a thin obstacle on the motion of a three-dimensional incompressible viscous flow. 
More precisely, for any  {\it obstacle} $\Cc_n$, i.e. verifying 
\begin{equation}\label{patate}
\text{$\Cc_n$ is a compact subset of $\R^3$ such that $\R^3\setminus \Cc_n$ is simply connected,}
\end{equation}
we consider the 3D-Navier-Stokes equations on $\OM_n:=\R^3\setminus \Cc_n$:
\begin{equation}\label{NS1}
\partial_t u^n -\nu \Delta u^n+u^n \cdot\nabla u^n=-\nabla p^n \quad \forall (t,x)\in (0,\infty)\times \OM_n,
\end{equation}
where $u^n=(u_1^n,u_2^n,u_3^n)$ denotes the velocity, $p^n$ the pressure and $\nu$ the viscosity. The incompressibility and the no-slip boundary conditions reads
\begin{equation}\label{NS2}
\diver u^n = 0 \quad \forall (t,x)\in [0,\infty)\times \OM^n, \qquad u^n=0  \quad \forall (t,x)\in (0,\infty)\times \pd\OM^n.
\end{equation}

A natural quantity for incompressible flows is the vorticity:
\begin{equation*}
\omega^n:=\curl u^n = (\partial_2 u_3^n-\partial_3u_2^n,\partial_3 u_1^n-\partial_1u_3^n,\partial_1 u_2^n-\partial_2u_1^n).
\end{equation*}
As the domains $\OM_n$ depend on $n$, it is standard  to give an initial condition in terms of a vorticity independent of $n$ (see \cite{ift_lop_NS,lac_NS,ift_kell}): for any smooth initial vorticity $\om_0$ which is divergence free and compactly supported in $\R^3$,  Lemma \ref{coro} states that there exists a unique vector field $u_0^n$ on $\OM_n$ such that:
\begin{equation}\label{u0n}
\diver u_0^n = 0, \quad \curl u_0^n = \om_0\vert_{\OM_n}, \quad u_0^n\cdot \hat n\vert_{\pd \Cc_n}=0, \quad u_0^n \in L^2(\OM_n).
\end{equation}
For such an initial velocity, it is well known that there exists a global weak solution $u^n$ of the Navier-Stokes equations \eqref{NS1}-\eqref{NS2} on $\OM_n$ in the  sense of Leray (see Definition \ref{defNS}). 

{\it The purpose of this paper is to study the asymptotic behavior of $u^n$ when $\Cc_n$ shrinks to a curve or a surface}.

\subsection{Leray solutions}

In order to give precisely the main theorems, we recall here the notion of weak solution in the sense of Leray.
We begin by introducing the classical solenoidal vector fields spaces.
\begin{definition} \label{spaces} 
Let $\OM$ an open subset of $\R^3$. 
We denote by 
\begin{itemize}
\item $V (\OM):= \Bigl\{ \varphi \in C_0^\infty(\OM) \: \vert \: \diver \varphi = 0 \text{ in } \OM \Bigl\}$;
\item $\Hc(\OM)$ the closure of $V (\OM)$ in the norm $L^2$;
\item $\Vc(\OM)$ the closure of $V (\OM)$ in the norm $H^1$, and its dual space by $\Vc'(\OM)$;
\item $\displaystyle H_2(\OM):= \Bigl\{ \varphi \in L^2(\OM) \: \vert \: \diver \varphi = 0 \text{ in } \OM,\: \varphi\cdot n =0 \text{ at } \pd\OM \Bigl\}$\footnote{$\varphi\cdot n$ should be understood in $H^{-1/2}(\pd \OM)$ (see e.g. \cite[Theo III.2.2]{galdi}).} ;
\item $\displaystyle G(\OM):= \Bigl\{ w \in L^2(\OM) \: \vert \: w = \na p \text{, for some } p\in H^1_{\loc}(\OM) \Bigl\}$;
\end{itemize}
\end{definition}

For any arbitrary domain $\OM$ in $\R^3$, we know that $G(\OM)$ and $\Hc(\OM)$ are orthogonal subspaces in $L^2(\OM)$ (see e.g.  \cite[Theo III.1.1]{galdi}). Moreover
\begin{equation}\label{leray projection}
 L^2(\OM) = G(\OM) \oplus \Hc(\OM),
\end{equation}
which implies  the existence of a unique projection operator (called Leray projection):
\[ \Pb_{\OM} \: : \: L^2(\OM) \to \Hc(\OM),\]
and by orthogonality we have:
\[ \| \Pb_{\OM} u \|_{L^2} \leq \|u \|_{L^2}, \quad \forall u \in L^2(\OM).\]

Let us also mention that \eqref{leray projection} implies the existence and uniqueness of a solution of \eqref{u0n}:
\begin{lemma}\label{coro}
Let $\Cc_n$ be a smooth obstacle (in the sense of \eqref{patate}) and $\OM_n=\R^3\setminus \Cc_n$. If $\om_0 \in V(\R^3)$, then there exists a unique solution of \eqref{u0n}. Moreover, 
\begin{equation}\label{Pnv0}
 u_0^n = \Pb_{\OM_n} v_0,
\end{equation}
where $v_0$ is  the Biot-Savart law in $\R^3$:
\begin{equation} \label{v_0_3D}
v_0(x)=-\int_{\R^3} \frac{x-y}{4\pi |x-y|^3}\times \om_0(y)\, dy,
\end{equation}
\end{lemma}
In \eqref{v_0_3D}, $\times$ denotes the standard cross product of vectors in $\R^3$. We note that $v_0$ is the unique vector field in $\R^3$ verifying:
\begin{equation*}
\diver v_0 = 0, \quad \curl v_0 = \om_0,  \quad v_0 \in L^2(\R^3).
\end{equation*}
For sake of completeness, this lemma is proved in Appendix \ref{app coro}.

Now we can give the definition of a global weak solution of the Navier-Stokes equations in the Leray sense.

\begin{definition}\label{defNS} Let $u_0 \in \Hc(\OM)$. We say that $u$  is a global weak solution of the Navier-Stokes equations on  $\OM$ with initial velocity $u_0$ iff 
\begin{itemize}
\item $u$ belongs to
\[ C([0,\infty);\Vc'(\OM)) \cap L^\infty_{\loc}([0,\infty); \Hc(\OM)) \cap L^2_{\loc}([0,\infty);\Vc(\OM)) ;\]
\item $u$ verifies the momentum equation in the sense of $\Vc'(\OM)$, i.e. $\forall \p\in C^1([0,\infty); \Vc(\OM))$, we have for all $t$:
\begin{equation}\label{eq3D}
\int_{\OM} (u\cdot \p)(t,x)\, dx+  \int_0^t \int_{\OM}( -u\cdot \p_t+ \nu \na u : \na \p - (u\otimes u ): \na \p)(t',x) \,dx\,dt'=\int_{\OM} u_0\cdot \p(0,\cdot);
\end{equation}
\item $u$ verifies the energy inequality:
\begin{equation}\label{est_vit}
\|u(t)\|^2_{L^2(\OM)} + 2\nu \int_0^t \|\na u(t)\|^2_{L^2(\OM)} \leq \|u_0\|^2_{L^2(\OM)}\ \ \forall t\geq 0.
\end{equation}
\end{itemize}
\end{definition}

Without any assumption about the regularity of $\OM$, the Leray theorem states that there exists a global weak solution of the Navier-Stokes equations in the sense of the previous definition (see e.g. \cite[Theo 2.3]{gallagher} and \cite[Theo III.3.1]{temam}). 

\subsection{Main results}

In Section \ref{sect 2}, we establish that the Navier-Stokes equations is structurally stable under Hausdorff approximations of the fluid domains:

\begin{theorem}\label{main 1}
Let $\Cc$ be an obstacle of $\R^3$ (in the sense of \eqref{patate}) which is a limit (in the Hausdorff sense) of  a sequence of smooth obstacles $\{ \Cc_n \}$. Let $\om_0\in V(\R^3)$  and $u^n$ be a global weak solution to the Navier-Stokes equations on $\OM_n=\R^3\setminus\Cc_n$ (in the sense of Definition \ref{defNS}) with initial velocity $u_0^n$ (given by \eqref{Pnv0}, which is the solution of \eqref{u0n}). Then we can extract a subsequence such that $Eu^n$ converges weakly-$*$ to $u$ in $L^\infty(\R^+;L^2(\OM))$, where $u$ is a global weak solution of the Navier-Stokes equations on $\OM:=\R^3\setminus \Cc$ with initial velocity $u_0=\Pb_{\OM} v_0$.
\end{theorem}

Here, $E u^n$ is the extension of $u^n$ on $\R^3$, vanishing on $\Cc_n$, and $\Cc_n$ converges to $\Cc$ in the Hausdorff sense if and only if the Hausdorff distance between $\Cc_n$ and $\Cc$ converges to zero. See for example \cite[Appendix B]{GV_lac} for more details about the Hausdorff topology, in particular the Hausdorff convergence implies the following proposition: 
\begin{equation}\label{hausdorff}
\text{for any compact set }K\subset \Omega, \text{ there exists }n_{K}>0 \text{ such that }K\subset \Omega_{n}, \ \forall n\geq n_{K}.
\end{equation}
Actually, we will also show that for any sequence $u_0^n\in \Hc(\OM_n)$ which verifies $Eu_0^n \to u_0 \in \Hc(\OM)$ (for the $L^2$ norm), then we can extract a subsequence such that $Eu^n$ converges weakly-$*$ to $u$ in $L^\infty(\R^+;L^2(\OM))$, where $u$ is a global weak solution of the Navier-Stokes equations on $\OM=\R^3\setminus \Cc$ with initial velocity $u_0$. Of course, to pass to the limit in the non-linear term, we will need a strong compactness argument in $L^2_{\loc}(\R^+\times \OM)$. However, the precise statement is not very convenient to give  here. Indeed, we will decompose the velocity in two parts (depending on the compact subset of $\OM$) and we will prove the strong compactness only of one part of $u^n$ (see Subsection \ref{time evolution} for more details).

\medskip

More importantly, we wonder for which condition $\Cc$ is removable, i.e. $u$ is the solution of the Navier-Stokes on the full space $\R^3$. Such an issue presents a large literature on experiments and simulations (see  e.g. \cite{phy_1,phy_2,phy_3,phy_4,phy_5,phy_6}  and references therein). As the solution $u$ belongs to $H^1_0(\OM)$ for a.e. time, the natural notion is the Sobolev $H^1$ capacity of the obstacle, which is defined by 
$$ {\rm cap}(\Cc) \: := \: \inf \{ \| v \|^2_{H^1(\R^N)}, \: v \ge 1 \: \mbox{ a.e.    in a neighborhood of } \Cc\}.
$$
The capacity is not a measure, but has similar good properties. For nice sets  $E$ in $\R^N$, the capacity of $E$ can be thought very roughly as some $n-1$ dimensional Hausdorff measure of its boundary. More precisely:
\begin{enumerate}
\item For all compact  set $K$ included in a bounded open set $D$, \\ 
$\: {\rm cap}(K) = {\rm cap}(\partial K)$.
\item If $E \subset \R^N$ is contained in a manifold of dimension $N-2$, then ${\rm cap}(E) = 0$.
\item  If $E \subset \R^N$ contains a piece of some smooth hypersurface (manifold of dimension N-1), then ${\rm cap}(E) > 0$. 
\item Let $D$ and $\OM$ be open sets such that $\OM \subset D$. Then
\[ \Bigl( v\in H^1_0(\OM)\Bigl) \iff \Bigl( v\in H^1_0(D) \text{ and } v = 0 \text{ quasi everywhere in } D\setminus \OM \Bigl),\]
which means that $v=0$ except on a set with zero capacity.
\end{enumerate}
We refer to \cite{henrot} for all details on the Sobolev capacity (see \cite[Appendix A]{GV_lac} for a short summary).

If $\Cc$ is a compact subset of $\R^3$ which contains  a piece of smooth hypersurface, then we infer from (3) and (4) that for a.e. $t$, $u$ vanishes quasi everywhere in $\Cc$, where ${\rm cap}(\Cc)>0$. Therefore, a surface is not removable  for the 3D viscous fluid, as a curve for the 2D viscous fluid (case treated in \cite{lac_NS}).

Then, we turn to the  obstacles with zero capacity. Iftimie, Lopes Filho and Nussenzveig Lopes have considered in \cite{ift_lop_NS} the 2D case  where one obstacle shrinks homotetically to a point, whereas Iftimie and Kelliher show in \cite{ift_kell} that a point in 3D has no influence on a viscous fluid.
There is a small restriction in dimension two, due to the fact that the exterior of one obstacle is not simply connected. In this case the vorticity is not sufficient to determine uniquely the velocity, and we have also to prescribe the initial circulation of $u_0^n$ around $\Cc_n$. Nevertheless, if this initial circulation is assumed to be zero, \cite{ift_lop_NS} exactly states that a material point has no effect on a 2D fluid (see \cite{ift_lop_NS} for more details).

Section \ref{sect curve} is devoted to prove that a curve is a removable singularity for the 3D Navier-Stokes equations.

\begin{theorem}\label{main 2}
Let $\Cc$ be a $C^2$ injective compact curve of $\R^3$, and $\om_0\in V(\R^3)$. Let $\{ \Cc_n \}$ a family of smooth obstacles (in the sense of \eqref{patate}) converging to $\Cc$ (in the Hausdorff sense), such that $\Cc\subset \Cc_n$. 
Let $u^n$ be a global weak solution to the Navier-Stokes equations on $\OM_n=\R^3\setminus\Cc_n$  with initial velocity $u_0^n$ (given by \eqref{Pnv0}), then we can extract a subsequence such that $Eu^n$ converges weakly-$*$ to $u$ in $L^\infty(\R^+;L^2(\R^3))$, where $u$ is a global weak solution of the Navier-Stokes equations on $\R^3$ with initial velocity $v_0$.
\end{theorem}

Even if this theorem is the natural extension of \cite{ift_kell,ift_lop_NS}, the proof requires a new way to cut-off divergence free test functions. In the two previous articles, the classical cutting-off of stream function was sufficient to get the stability result. Such a method does not hold in our case, and we will introduce a non-explicit approximation, based on Bogovskii results.

\medskip

{\it Notation:} for any function $f$ defined on $\OM_n$, we denote by $E f$ the extension of $f$ on $\R^3$, vanishing on $\Cc_n$. If $f$ is regular enough and vanishes on $\pd \Cc_n$, then $\na (E f)=E(\na f)$. Similarly, if $v$ is a vector field regular enough and tangent to the boundary, then $\diver Ev = E(\diver v)$.

\bigskip

\section{Stability under Hausdorff approximations}\label{sect 2}

\subsection{Convergence of the initial velocity}

Let $\Cc$ be an obstacle of $\R^3$ (in the sense of \eqref{patate}) and $\{\Cc_n\}$ be a family of smooth obstacles converging to $\Cc$ in the Hausdorff topology when $n\to \infty$ and such that $\Cc \subset \Cc_n$.
As mentioned in the introduction, we fix $\om_0\in V(\R^3)$ and we define $u_0^n$ as in \eqref{Pnv0}, which is the unique vector field verifying \eqref{u0n}. Let us show that $u_0^n$ converges strongly to $u_0:=\Pb_{\OM} v_0$. 

\begin{proposition}\label{u_0_conv_3D} 
With the above notations, we have that 
\[ E{u_0^n}\to u_0 \text{ strongly in } L^2(\OM).\]
\end{proposition}

\begin{proof}  As the Leray projection is orthogonal in $L^2$, we get from \eqref{Pnv0} that 
\[ \| Eu_0^n \|_{L^2(\OM)}= \|u_0^n \|_{L^2(\OM_n)} \leq \| v_0 \|_{L^2(\OM_n)}\leq \| v_0 \|_{L^2(\OM)} .\]
By the Banach-Alaoglu's theorem, we infer that there exists $w_0 \in L^2(\OM)$ and a subsequence $n\to \infty$, such that 
\[ E u_0^n\rightharpoonup w_0 \text{ weak in } L^2(\OM).\] 
This weak convergence implies in particular that
\[  \| w_0 \|_{L^2}  \leq \liminf \| E u_0^n\|_{L^2},\]
and
\[ \curl w_0 = \om_0\vert_{\OM} \text{ in } \Vc'(\OM).\]
In the proof of Lemma \ref{coro} (see Appendix \ref{app coro}), we have proved that $\curl \mathbb{P}_{\OM_n} w_0 = \curl w_0 \vert_{\OM_n}=\om_0 \vert_{\OM_n}$. The uniqueness part of Lemma \ref{coro} implies that $ \mathbb{P}_{\OM_n} w_0=\mathbb{P}_{\OM_n} v_0=u_0^n$ and then
\[ \|E u_0^n\|_{L^2} \leq \| w_0 \|_{L^2}.\]
Putting together the two last inequalities, we get $\|E u_0^n\|_{L^2}\to \| w_0 \|_{L^2}$. Using the weak convergence in $L^2(\OM)$ of $Eu_0^n$ to $w_0$, we obtain the strong convergence in $L^2(\OM)$.

As $Eu_0^n$ belongs to $\Hc(\OM_n)$, we deduce directly from the $L^2$ strong convergence that $w_0$ belongs to $\Hc({\OM})$.  Then we have two functions in $\Hc({\OM})$ having the same vorticity, which implies that $u_0=w_0$, without assuming nothing about the regularity of $\Cc$ (see Appendix \ref{app coro}). The uniqueness also implies that the limit holds without extracting a subsequence.

\end{proof}

\begin{remark}\label{rem curl}
We have obtained in the previous proof that 
\[ \curl u_0= \om_0 = \curl v_0  \text{ in } \Vc'(\OM)\]
but we do not have
\[  \curl u_0 = \om_0 = \curl v_0  \text{ in } H^{-1}(\R^3).\]
Even in the case where $\Cc$ is a surface, we can just pretend that $\curl(u_0-v_0)$ belongs to $H^{-1}(\R^3)$ and is supported on the surface. Actually, we will prove in Subsection \ref{sect surface} that
\[ \curl u_0 = \om_0 + g_{\Cc} \d_{\Cc}\]
where $g_{\Cc}$ is the jump of the tangential component of $u_0$ and $\d_{\Cc}$ is the Dirac measure on $\Cc$.

In the case of the curve, we will show in Subsection \ref{sect initial curve} that there is no function belonging in $H^{-1}(\R^3)$ compactly supported on a curve, and we will obtain that $u_0=v_0$.\end{remark}

\subsection{Time evolution}\label{time evolution}

For all $n$, we denote by $u^n$ a global weak solution of the Navier-Stokes equations (in the sense of Definition \ref{defNS}) on $\OM_n=\R^3\setminus \Cc_n$ with initial data $u_0^n$.

By Proposition \ref{u_0_conv_3D}, we already know that $Eu_0^n \to u_0$ in $L^2(\OM)$. Moreover, thanks to the energy inequality  \eqref{est_vit}, we state that
\begin{equation}\label{estuni}
Eu^n \text{ is uniformly bounded in } L^\infty((0,\infty);\Hc_{\OM}) \text{ and } \na Eu^n \text{ is uniformly bounded in } L^2((0,\infty)\times \OM).
\end{equation}
Now, we need to establish a temporal estimate. First, we use that $u^n$ verifies the Dirichlet boundary condition for a.e. $t>0$ in order to write the following Sobolev inequality:
\begin{eqnarray}
\| u^n\|_{L^4(\OM_n)} &=& \| E u^n \|_{L^4(\R^3)} \leq \| E u^n\|_{L^2(\R^3)}^{1/4} \| Eu^n\|_{L^6(\R^3)}^{3/4} \nonumber\\
&\leq& C \| E u^n\|_{L^2(\R^3)}^{1/4} \| \na( Eu^n) \|_{L^2(\R^3)}^{3/4} = C \| u^n\|_{L^2(\OM_n)}^{1/4} \| \na u^n \|_{L^2(\OM_n)}^{3/4},  \label{sobolev}
\end{eqnarray}
where $C$ is independent of $n$. Let us consider $T>0$ and $O$ an open smooth bounded set relatively compact in $\OM$. By  \eqref{hausdorff}, there exists $n_O$ such that $O\cap \Cc_n=\emptyset$ for all $n\geq n_O$. Even if it is possible to show that $\{E u^n\}$ is equicontinuous in $\mathcal{V}'(O)$, we cannot deduce precompactness in $L^2([0,T]\times O)$: indeed there is no injection from $\{v\in L^2(O), \ \diver v=0\}$ to $\mathcal{V}'(O)$ (the gradients of harmonic function are divergence free, but their $\mathcal{V}'(O)$ norm are equal to zero). To use interpolation, we have to consider  $\Hc(O)$ which embeds in $\mathcal{V}'(O)$: the boundary condition forbids the gradients of harmonic function. Then, we write in $O$ 
\[
u^n = \mathbb{P}_{O} u^n + \nabla q^n, \quad \Delta q^n =0,
\]
with 
$$
\|  \mathbb{P}_{O} u_n \|_{L^2({O})}^2+ \|  \nabla q^n \|_{L^2({O})}^2 = \|  u^n \|_{L^2({O})}^2, \quad  \| \mathbb{P}_{O} u_n \|_{H^1({O})} + \|  \nabla q^n \|_{H^1({O})}\leq  C \|  u^n \|_{H^1({O})},$$
 which imply that $\nabla q^n$ converges to $\nabla q$ weak-$*$  in $L^\infty([0,T];L^2(O))$ and in $L^2([0,T];H^1(O))$. For $ \mathbb{P}_{O} u^n $, we perform a strong compactness argument as follows: for any $\F\in \mathcal{V} (O)$ and $n\geq n_O$:
\begin{eqnarray*}
|\langle \mathbb{P}_{O} u^n(t), \F\rangle-\langle \mathbb{P}_{O} u^n(s), \F\rangle |&=&\Bigl| \int_s^t\int_{\OM_n} -\nu \na u^n : \na \F + (u^n \otimes  u^n): \na \F\Bigl| \\
&\leq& \nu (t-s)^{1/2} \| \na u^n\|_{L^2(\R^+\times \OM_n)}  \| \na \F\|_{L^2} \\
&&+ (t-s)^{1/4} \| u^n \|_{L^\infty(\R^+,L^2(\OM_n))}^{1/2} \| \na u^n\|_{L^2(\R^+\times \OM_n)}^{3/2} \|\na \F\|_{L^2}\\
&\leq &C((t-s)^{1/2} + (t-s)^{1/4}) \| \F\|_{H^1},
\end{eqnarray*}
where we have used \eqref{sobolev}. This inequality implies that $\{\mathbb{P}_{O} u^n\}$ is equicontinous as a family of functions from $\R^+$ to $\Vc'(O)$. Using that this family is bounded in $L^\infty(\R^+;\Hc(O))$ and the compact embedding in $\Vc'(O)$, Ascoli theorem gives that $\{ \mathbb{P}_{O} u^n\}$ is precompact in $L^\infty((0,T);\Vc'(O))$. Moreover, $\{ \mathbb{P}_{O} u^n\}$ is also bounded in $L^2((0,T);\Vc(O))$, then we get by interpolation that this family is precompact in $L^2((0,T) \times O)$. By a diagonal extraction on the compact sets of $\OM$ and $[0,+\infty)$, we find a subsequence $Eu_n$ such that we have the following property: for any compact set $[0,T]\times K \subset [0,\infty)\times \OM$, then there exists $O$ a relatively compact set of $\OM$ belonging in the sequence where the diagonal extraction was considered, such that $K\subset O$ and $\mathbb{P}_{O} u_n$ is precompact in $L^2((0,T)\times K)$.

Moreover,  extracting again a subsequence if necessary, we know from \eqref{estuni} that  the limit verifies
\begin{equation}\label{u space}
u \in L^\infty((0,\infty);\Hc({\OM})) \text{ and } \na u \in L^2((0,\infty)\times \OM).
\end{equation}

For any test function $\p\in C^1(\R^+;\Vc(\OM))$, there exist $\p^k \in C^\infty_c((0,\infty)\times \OM)$, $\diver \p^k=0$ such that\footnote{see e.g. \cite[Prop 3.6]{lac_NS}.}
\[ \p^k \to \p \text{ strongly in } L^q_{\loc}(\R^+;\Vc(\OM)), \quad \forall q\in(1,\infty).\]
For $k$ fixed, there exists $O$ a relatively compact set of $\OM$, belonging in the sequence where the diagonal extraction was considered, where $\supp \p^k \subset O$. As there exists $n_O$ such that $O \cap \Cc_n = \emptyset, \ \forall n\geq n_O$, then \eqref{eq3D} reads
\begin{equation*}
\begin{split}
 0= \int_0^\infty \int_{\OM}\Bigl(& -u^n\cdot \p^k_t+ \nu \na u^n : \na \p^k - (u^n\otimes u^n ): \na \p^k\Bigl) \,dx\,dt'\\
 0=  \int_0^\infty \int_{\OM}\Bigl(& -u^n\cdot \p^k_t+ \nu \na u^n : \na \p^k - ( \mathbb{P}_{O} u^n\otimes  \mathbb{P}_{O} u^n ): \na \p^k - ( \mathbb{P}_{O} u^n\otimes  \nabla q^n ): \na \p^k \\
 & - (\nabla q^n\otimes  \mathbb{P}_{O} u^n ): \na \p^k - (\nabla q^n \otimes  \nabla q^n ): \na \p^k \Bigl)   \,dx\,dt'.
\end{split}
\end{equation*}
Actually, the last term is equal to zero: indeed we can check that $\diver (\nabla q^n \otimes  \nabla q^n )=\frac12 \nabla (|\nabla q^n |^2)$ and as it is a gradient, the last part vanishes by the divergence free condition of $\p^k$. Thanks to the strong convergence of $ \mathbb{P}_{O} u^n$ in $L^2_{\loc}(\R^+\times O)$ and the weak of $\na u^n$, $u^n$, $\nabla q^n$, we can pass to the limit $n\to \infty$ to get:
\[  \int_0^\infty \int_{\OM}( -u\cdot \p^k_t+ \nu \na u : \na \p^k - (u\otimes u ): \na \p^k) \,dx\,dt'=0.\]
In particular, this equality putting together with \eqref{u space} and \eqref{sobolev} gives that
\[ \pd_t u \text{ belongs to } L^{4/3}_{\loc}(\R^+,\Vc'(\OM)).\]
Then we can pass to the limit in
\begin{equation}\label{mom 1}
 \int_0^\infty \int_{\OM}( \pd_t u\cdot \p^k+ \nu \na u : \na \p^k - (u\otimes u ): \na \p^k) \,dx\,dt'=0
\end{equation}
as $k\to \infty$ to get
\[\forall \f \in C^\infty_c(0,\infty),\quad  \int_0^\infty \int_{\OM}( \pd_t u\cdot \p\f+ \nu \f \na u : \na \p - \f(u\otimes u ): \na \p) \,dx\,dt'=0.\]
This equality implies that
\begin{equation}\label{mom bis}
 \frac{d}{dt}  \int_{\OM}u\cdot \p = \int_{\OM} ( u\cdot \p_t - \nu \na u : \na \p + (u\otimes u ): \na \p) \,dx
 \end{equation}
in the sense of distribution in $\R^+$. Since the right hand side term belongs to $L^1_{\loc}(\R^+)$, the equality holds in $L^1_{\loc}(\R^+)$.

\medskip

Now, we check that $u$ belongs in the good functional space. Thanks to the previous equality, and \eqref{u space}, then we can easily prove that
\[ u \in C([0,\infty); \Vc'(\OM)) \cap C_w([0,\infty); \Hc(\OM)).\]
This argument can be found in \cite[Subsection III.3.1]{temam}: as $\pd_t u$ belongs to $L^1(\Vc')$, then its implies that $u$ is almost everywhere equal to a function continuous from $\R^+$ into $\Vc'$. Moreover, using the fact that $u\in L^\infty(\Hc)$, then \cite[Lem 1.4]{temam} states that the continuity in $\Vc'$ implies the weak continuity  in time with values in  $\Hc$.

Moreover, thanks to the continuity in $\Vc'$, we infer that the equality \eqref{mom bis} in the sense of $L^1_{\loc}(\R^+)$ implies that the integral equality \eqref{eq3D} holds for all $t>0$. Indeed, for the initial data we know from the uniform convergence in $H^{-2}_{\loc}(\OM)$ that $Eu_0^n\to u\vert_{t=0}$ in $H^{-2}_{\loc}(\OM)$. However, we proved in Proposition \ref{u_0_conv_3D} that $Eu_0^n \to u_0$ in $L^2(\OM)$, which allows us to state by the uniqueness of the limit in $H^{-2}_{\loc}$ that the initial velocity is $u_0=\Pb_{\OM}v_0$.

\medskip

To finish the proof of Theorem \ref{main 1}, we have to prove the energy inequality. We take the liminf of \eqref{est_vit}:
\[\liminf_{n\to 0} \|E u^n(t)\|^2_{L^2(\OM)} + 2\nu \liminf_{n\to 0} \int_0^t \|\na E u^n(t)\|^2_{L^2(\OM)} \leq \|u_0\|^2_{L^2(\OM)}\]
and we note that the weak limit\footnote{Indeed, by the uniform estimates and diagonal extraction, we can find a common subsequence such that $E u^n(t)$ weakly converge for all $t\in \mathbb{Q}^+$. Then, we conclude by the continuity that this sequence holds for all $t\in \R^+$.} in $L^2$ of $E u^n(t)$ to $u(t)$ and the weak limit in $L^2((0,t)\times \OM)$ of $\na E u^n$ to $\na u$ imply that
\[ \|u(t)\|^2_{L^2(\OM)} \leq \liminf_{n\to 0} \|E u^n(t)\|^2_{L^2(\OM)} \text{ and }\|\na u\|^2_{L^2((0,t)\times \OM)} \leq \liminf_{n\to 0} \|\na E u^n\|^2_{L^2((0,t)\times \OM)}.\]
It gives the last point required in Definition \ref{defNS}, which ends the proof of Theorem \ref{main 1}.

\bigskip

Therefore, we have shown that the Navier-Stokes solutions converge when the smooth obstacles convergence to a obstacle $\Cc$ verifying \eqref{patate}. We note here that we do not assume any assumption on the regularity of $\Cc$. In particular, this result holds if $\Cc$ is a surface.  As $u(t,\cdot) \in \Vc(\OM)$ for a.e. time, it is clear that the surface has a non-negligible effect on the motion of 3D viscous flow: $u$ verifies for almost every time the no slip boundary condition. In the following subsection, we discuss about the initial velocity properties in the particular case of the surface, and the goal is to get some  similarities  to the curve in 2D.

\subsection{Remark on the behavior of the initial velocity near a smooth surface}\label{sect surface}

In the two dimensional case, we obtain in \cite{lac_NS} an explicit formula of the initial velocity in terms of Riemann maps (identifying $\R^2$ and $\C$).  This formula allows us to state that $u_0$ is continuous up to the curve with different values on each side, except near the end-points where it behaves as the inverse of the square root of the distance. In our case, we do not have this formula, and we use in this subsection classical elliptic theory in order to get similar results in the case where $\Cc$ is a bounded orientable surface of codimension 1 in $\R^3$.

As mentioned in Remark \ref{rem curl}, $u_0-v_0$ is curl free in $\OM$ and as $\OM$ is simply connected we infer that there exists $p$ such that $u_0-v_0=\na p$. We know that $v_0$ is continuous up to the boundary, then the goal is to determine the behavior of $\na p$ near $\Cc$ where $p$ verifies the following elliptic problem
\begin{equation*}
\left\lbrace\begin{aligned}
&\D p=0 &\text{ in } \OM \\
&\frac{\pd p}{\pd n}=-v_0\cdot n &\text{ on } \Cc,
\end{aligned}\right.
\end{equation*}
where $v_0$ is regular on $\Cc$.

This subsection is independent of the convergence theory, and the goal here is  to give an example of behavior of $u_0$. Therefore, we add here some assumptions on $\Cc$ in order to apply classical elliptic results.
We assume that $\Cc$ is a $C^\infty$ manifold and its boundary $\Bc$ is a $C^\infty$ closed curve. For example, if $\Cc=\{ (x,y,0)\in \R^3,\ x^2+y^2 \leq 1\}$ we denote by $\Bc=\{ (x,y,0)\in \R^3,\ x^2+y^2 =1\}$.

The study of elliptic equations in the exterior of a surface with the Neumann condition is standard for the crack problem in 3D linear elasticity. Actually, to get exactly the Neumann boundary condition, we add a regular function $h$ such that $\tilde p:=p+h$ verifies
\[ \D \tilde p= f \text{ in }\OM, \quad \frac{\pd \tilde p}{\pd n}=0 \text{ on }\Cc.\]
For $\om_0$ regular enough, we rich the necessary regularity for $f$ in order  that $\tilde p$ has an expansion near $\Bc$ on the form
\[ \sum_{k\geq 0} r^{\frac12+k} \p(\th)\]
in local polar coordinates $(r,\th)$. Such a result is proved in \cite{dauge} (see also the references therein). In particular, it implies that $u_0$ is continuous up to $\Cc$, with possibly different values on each side, except near the boundary $\Bc$ where $u_0$ behaves like the inverse of the square root of the distance. Therefore, we obtain exactly the same behavior in 3D in the exterior of a surface than in 2D in the exterior of a curve.

Thanks to the continuity up to the surface, it is easy to see that the tangent condition implies that
\[ \diver u_0 = 0 \text{ and } \curl u_0=\om_0+g \d_\Cc\]
in $\Vc'(\R^3)$, where $g$ is the jump of the tangential component of $u_0$ through the surface (see e.g. \cite[Lem 5.8]{lac_euler} for this proof in dimension two).


\section{Viscous flow around a curve}\label{sect curve}

As in the previous section and as in \cite{ift_kell}, we consider $\{\Cc_n\}$ a family of smooth obstacles of $\R^3$ (in the sense of \eqref{patate})  which converges to $\Cc$ in the Hausdorff topology when $n\to \infty$. Here, $\Cc$ is assumed to be a compact injective $C^2$ curve  in $\R^3$ (i.e. included in a smooth manifold of dimension 1). The goal of this section is to prove that the curve is a removable singularity for the Naviers-Stokes solutions in $\R^3$.

\subsection{Convergence of the initial velocity}\label{sect initial curve}

We will need of a suitable cutoff function of a small neighborhood of $\Cc$. Let $\chi$ be a function verifying
\[ \chi \in C^\infty(\R), \quad \chi(s) \equiv 0 \text{ on } (-\infty,1), \quad \chi(s) \equiv 1 \text{ on } (2,+\infty), \]
then we define
\begin{equation}\label{cutoff}
 \y_n(x) := \chi \Bigl( n {\rm d}(x,\Cc) \Bigl).
\end{equation}

Although it is obvious that $\y_n$ vanishes in a small neighborhood of $\Cc$, we need the following estimates of $\| \na \y_n \|_{L^p}$.

\begin{proposition}\label{cutoff est}
Let $\Cc$ be a $C^2$ compact injective curve.
 There exist $n_0$ and $C>0$ such that for all $n> n_0$, 
\[\| \na \y_n \|_{L^\infty(\R^3)}\leq Cn \text{ and } {\rm meas}\Bigl(  \supp{(1- \y_n)}\Bigl) \leq C/n^2,\]
where {\rm meas} is the Lebesgue measure.
\end{proposition}

Even if such a proposition seems standard, we write the details in Appendix \ref{appendix} for  sake of completeness.
Thanks to this proposition, we can prove that there is no function supported on $\Cc$ which belongs to $H^{-1}(\R^3)$.

\begin{lemma}
Let $f$ a function belonging in $H^{-1}(\R^3)$. If
\[ \langle f, \f \rangle_{H^{-1},H^1} = 0, \ \forall \f \in C^\infty_c(\R^3 \setminus \Cc)\]
then $f=0$ in $\Dc'(\R^3)$.
\end{lemma}
\begin{proof} 
We fix $\f \in C^{\infty}_c(\R^3)$, and we introduce 
\[ \f_n(x):= \y_n(x) \f(x)\]
where $\y_n$ is the cutoff function defined in \eqref{cutoff}. As $\y_n$ is $C^1$ (see Appendix \ref{appendix}) and supported in the exterior of a small neighborhood of $\Cc$, then by assumption (and density of $C^\infty_c(\R^3 \setminus \Cc)$ in $C^1_c(\R^3 \setminus \Cc)$  for $H^1(\R^3)$ norm), we have that
\[ \langle f, \f_n \rangle_{H^{-1},H^1} = 0, \ \forall  n.\]
Thanks to Proposition \ref{cutoff est}, we get easily that
\[ \|\f_n-\f \|_{L^2}\leq C(\f) / n \to 0 \text{ and } \|\na \f_n-\na \f \|_{L^2}\leq C(\f).\]
By Banach-Alaoglu theorem, we can extract a subsequence such that $\f_n - \f \rightharpoonup 0$ weakly in $H^1(\R^3)$. 
In particular, it implies that
\[ \langle f,\f \rangle = \langle f,\f_n \rangle+\langle f,\f-\f_n \rangle= \langle f,\f-\f_n \rangle \to 0\]
as $n\to \infty$. Then, we have proved that $ \langle f,\f \rangle=0$, for all $\f \in C^{\infty}_c(\R^3)$.
\end{proof}

Now, we can come back to our main problem. Let $\Cc$ be an injective compact smooth curve in $\R^3$. We consider $\{\Cc_n\}$ a family of smooth obstacles of $\R^3$ such that $\Cc\subset \Cc_n$ and $\Cc_n$ converges to $\Cc$ in the Hausdorff topology when $n\to \infty$. Let $\om_0\in V(\R^3)$ be an initial vorticity. As in the previous section we set 
\[ \om_0^n:= \om_0\vert_{\OM_n} \text{ with } \OM_n= \R^3\setminus \Cc_n\]
then we define by $u_0^n:=\Pb_{\OM_n} v_0$ the unique vector field in $\OM_n$ solving \eqref{u0n}. We prove in the following proposition that we do not feel the presence of the curve for the limit initial velocity.

\begin{proposition}\label{u_0_conv_5.3}
With the above notation, we have 
\[E{u_0^n}\to v_0 \text{ strongly in } L^2(\R^3),\]
where $v_0$ is the velocity field without obstacle \eqref{v_0_3D}.
\end{proposition}

\begin{proof} As $v_0$ is continuous and behaves like $\mathcal{O}(1/|x|^2)$ at infinity (see \eqref{v_0_3D} for the explicit formula), we obtain directly that $v_0$ belongs to $L^p(\R^3)$ for any $p\in (3/2,\infty]$.

First, we introduce the stream function corresponding to $v_0$:
\[ \p(x)=-\int_{\R^3}\frac{x-y}{4\pi|x-y|^3}\times v_0(y) \, dy ,\]
which verifies $\curl \p = v_0$, $\diver \p =0$ and $\| \p\|_{L^\infty}\leq C \|v_0\|_{L^2\cap L^4}$ (see e.g. \cite{ift_kell}), so $\p$ is bounded.

Next, we define
\[ w^n : = \curl (\y_n \p) = \y_n v_0 + \na \y_n \times \p,\]
where $\y_n$ is the cutoff function \eqref{cutoff}. By construction, we deduce that $w^n$ belongs to $H_2(\OM_n)=\Hc(\OM_n)$ (see Appendix \ref{app coro}).

Finally, we use that $u_0^n$ is the $L^2$ projection of $v_0 \vert_{\OM_n}$ on $\Hc(\OM_n)$, to compute
\begin{eqnarray*}
\|E u_0^n -v_0\|_{L^2(\R^3)}&\leq& \| u_0^n -v_0\|_{L^2(\OM_n)}+\|v_0\|_{L^2(\Cc_n)} \\
&\leq& \| w^n -v_0\|_{L^2(\OM_n)}+\|v_0\|_{L^2(\Cc_n)}\\
&\leq& \|\na \y_n \times \p \|_{L^2(\OM_n)} + \|(1-\y_n) v_0\|_{L^2(\OM_n)} + \|v_0\|_{L^2(\Cc_n)} \\
&\leq& \|\na \y_n\|_{L^2} \|\p \|_{L^\infty}+ \|v_0\|_{L^\infty} (1/n+ {\rm meas}(\Cc_n))\\
&\leq& C,
\end{eqnarray*}
where we have used Proposition \ref{cutoff est}. This inequality implies that there exist $v\in L^2(\R^3)$ and a subsequence such that
\[E u_0^n - v_0 \rightharpoonup v \text{ weakly in }L^2.\]
By passing to the limit, we obtain that $\diver v= \curl v =0$ for all test function in $C^{\infty}_c(\R^3\setminus \Cc)$. Moreover, $\diver v$ and $\curl v$ belongs to $H^{-1}(\R^3)$ and we apply the previous lemma to state that $\diver v= \curl v \equiv 0$ on $\R^3$, and then $v=0$.

Moreover, reasoning as in the proof of Proposition \ref{u_0_conv_3D}, we pass from the weak convergence to the strong convergence of $u_0^n$ to $v_0$, which ends this proof.
\end{proof}

\subsection{Proof of Theorem \ref{main 2}}

The begin of the proof follows the same idea of the proof of Theorem \ref{main 1}: we prove that we have a strong limit of $Eu^n$ to $u$ a solution of the Navier-Stokes equations in ``the exterior of the curve''. The second step is to show that $u$ is actually a solution in $\R^3$.

\medskip

For all $n$, we consider $u^n$ a global weak solution, in the sense of Definition \ref{defNS}, of  \eqref{NS1}-\eqref{NS2} in $\OM_n$ with initial datum $u_0^n$, which is defined in the previous subsection.

By Proposition \ref{u_0_conv_5.3}, we already know that $Eu_0^n \to v_0$ in $L^2(\R^3)$. Next, we exactly follow Subsection \ref{time evolution}:
\begin{itemize}
\item uniform estimates of $Eu^n$ in $L^\infty((0,\infty);L^2(\R^3))$ and of $\na Eu^n$ in $L^2((0,\infty)\times \R^3)$;
\item equicontinuity in $\Vc^{-1}(O)$, with $O$ relatively compact in $\R^3 \setminus \Cc$;
\item precompactness of $\{\mathbb{P}_O u^n\}$ in $L^2((0,T) \times O )$;
\item by a priori estimates we know that
\begin{equation*}
u \in L^\infty((0,\infty);L^2(\R^3)) \text{ and } \na u \in L^2((0,\infty)\times \R^3);
\end{equation*}
\item passing at the limit $n\to \infty$ in the weak formulation of the momentum equation: $\forall \p^k\in C^\infty_0((0,\infty)\times(\R^3\setminus \Cc))\, \vert\, \diver \p^k=0$, we have \eqref{mom 1}.
\end{itemize}

\medskip

Now, we need to prove that the momentum equation is verified for all $\p \in C^1(\R^+;\Vc(\R^3))$. Then, to finish the proof of Theorem \ref{main 2}, we have to establish \eqref{eq3D} for test functions whose the support meets the curve. The following lemma will be the key of this extension.

\begin{lemma}
For all $\f\in C^\infty_c((0,\infty)\times \R^3)$ such that $\diver \f = 0$, then there exists a sequence $\p_n \in C^\infty_c((0,\infty)\times (\R^3\setminus\Cc))$ such that
\[ \diver \p_n =0 \text{ and } \p_n\rightharpoonup \f \: \text{ weak-$*$ in } L^\infty(\R^+; H^1(\R^3)).\]
\end{lemma}
\begin{proof} All the difficulty comes from the condition $\diver \p_n =0$. Indeed, without this condition, it is sufficient to multiply by the cutoff function. In \cite{ift_lop_NS,ift_kell}, the standard way to construct divergence free functions compactly supported outside the obstacle is to multiply the stream function by the cutoff function. However, we see in the proof of Proposition \ref{u_0_conv_5.3} that the computation of the $H^1$ norm makes appear $\| \na^2 \y_n \|_{L^2}$ which blows up strongly in our case.

So, we present here a new way to approximate divergence free function, which is not explicit as in the standard way. The following method comes from \cite[Chap III.4]{galdi} and is based on the Bogovskii operator. We set
\begin{equation}\label{fn}
 f_n(t,x):=\y_n(x) \f(t,x)
\end{equation}
which belongs to $C^\infty_c ((0,\infty);C^1_c(\R^3\setminus \Cc))$ but which is not divergence free. Its divergence
\[ g_n(t,x):= \diver f_n(t,x)= \f(t,x)\cdot \na \y_n(x)\]
verifies the following estimates for all $t$:
\[ \int_{\R^3} g_n(t,\cdot)=0 \text{ and } \|g_n(t, \cdot)\|_{L^p(\R^3)} \leq C n^{1-\frac2p},\]
where $C$ depends only on $\f$. To correct this divergence, we use the result of Bogovskii. Let $B$ be a ball big enough containing the support of $\f$. As the set $\widetilde \OM:= B\setminus \Cc$ verifies the cone condition (see \cite[Rem III.3.4]{galdi} for a precise definition), and as $g_n\in L^p(\widetilde\OM)$ with $\int_{\widetilde\OM} g_n=0$, then Theorem 3.1 in  \cite[Chap III.3]{galdi} states that there exists at least one solution $h_n$ of the following problem:
\[ \diver h_n = -g_n,\quad h_n\in W^{1,p}_0(\widetilde\OM), \quad \|h_n \|_{W^{1,p}(\widetilde\OM)} \leq c_p \| g_n \|_{L^p(\widetilde\OM)}.\]
Moreover, the constant $c_p$ depends only on $\widetilde\OM$ and $p$, and as $g_n$ as a compact support in $\widetilde\OM$, so is $h_n$. Extending by zero outside $\widetilde\OM$, we can define $h_n(t,\cdot)\in W^{1,p}(\R^3)$  for all $t$, and we 
have that $h_n\to 0$ strongly in $L^\infty(\R^+;W^{1,p}(\R^3))$ for all $p<2$ and uniformly bounded in $L^\infty(\R^+;H^1(\R^3))$. So we can extract a subsequence such that $h_n$ converges weak-$*$ in $L^\infty(\R^+;H^1(\R^3))$ and by uniqueness of the limit in $L^\infty(\R^+;W^{1,p}(B))$ it converges to $0$.

On the other hand, we know from the definition \eqref{fn} that $f_n-\f$ converges to $0$ strongly in $L^\infty(\R^+;W^{1,p}(\R^3))$ for any $p<2$ and weak-$*$ in $L^\infty(\R^+;H^1(\R^3))$.

In conclusion, $\p_n:= f_n + h_n\in L^\infty(\R^+;H^1(\R^3))$ is divergence free, compactly supported in $B\setminus \Cc$ and converges to $\f$ weak-$*$ in $L^\infty(\R^+;H^1(\R^3))$. Smoothing $\p_n$ by some mollifiers, we obtain the result.
\end{proof}

In the previous proof, we see that the approximation constructed verifies also
\[ \pd_t \p_n\rightharpoonup \pd_t \f \: \text{ weak-$*$ in } L^\infty(\R^+; H^1(\R^3)).\]
Therefore, we apply the previous lemma in order to pass to the limit in \eqref{mom 1}, which implies that this equality is verified for all test function in $C^\infty_c((0,\infty)\times \R^3)$. Now, we finish as we did in Subsection \ref{time evolution}:
\begin{itemize}
\item \eqref{mom 1} gives that $\pd_t u \in L^{4/3}_{\loc}(\R^+;\Vc'(\R^3))$;
\item we pass at the limit to say that the momentum equation is verified for all $\p\in C^1(\R^+;\Vc(\R^3))$ in $L^1_{\loc}(\R^+)$;
\item we get the continuity in time with values in $\Vc'(\R^3)$;
\item we conclude to the validity of \eqref{eq3D} for all $t$;
\item identification of the initial velocity: $u(0,\cdot)=v_0$ by Proposition \ref{u_0_conv_5.3};
\item thanks to weak convergence, we prove the energy inequality \eqref{est_vit} on $\R^3$.
\end{itemize}

Its ends to prove that $u$ is a global weak solution of the Navier-Stokes solutions in $\R^3$, in the sense of Definition \ref{defNS}.

\appendix

\section{Proof of Lemma \ref{coro}}\label{app coro}

First, we note from \cite[Theo III.2.2]{galdi} that the assumption on $\OM_n$ (smoothness and compact boundary) implies the coincidence of $\Hc(\OM_n)$ and $H_2(\OM_n)$.

By \eqref{leray projection}, we know that $v_0 \vert_{\OM} - \Pb_{\OM_n} v_0$ is a gradient and then $\curl \Pb_{\OM_n} v_0 = \curl v_0|_{\OM_n} =  \om_0 \vert_{\OM_n}$. Therefore $\Pb_{\OM_n} v_0$ is a solution of \eqref{u0n}, which proves the existence part of Lemma \ref{coro}.

Concerning the uniqueness, let us assume that $u$ and $w$ are two solutions. As $\OM_n=\R^3\setminus \Cc_n$ is simply connected in $\R^3$,  $\curl (u-w)=0$ on $\OM_n$ implies that there exists $p$ such that $u-w=\na p$. Then $\na p$ belongs to $\Hc_{\OM_n}$ which is possible only for $\na p=0$ (see \eqref{leray projection}). Its ends the proof of Lemma \ref{coro}.

\section{Cutoff functions}\label{appendix}

In order to differentiate the function $x\mapsto {\rm d}(x,\Cc)$, we have to check that the minimal distance ${\rm d}(x,y)$ is reached for a unique $y\in \Cc$, at least for $x$ closed enough to $\Cc$. We need to assume that the curve is at least $C^1$, because such a property is false near a corner. For example, if $\Cc$ is the curve $\{(x,y,0) \vert x\in [-1,1], y=|x|\}$ then for all $M$ on the half line $\{ (0,y,0) \vert y>0\}$, we have two minimums on $\Cc$: $M_1:=(-y/\sqrt2,y/\sqrt2,0)$ and $M_2:=(y/\sqrt2,y/\sqrt2,0)$ and ${\rm d}(M,\Cc)={\rm d}(M,M_1)={\rm d}(M,M_2)=y/\sqrt2$.

This example shows that we have to assume some regularity for the curve.

\begin{lemma}\label{unique-min}
Let $\Cc$ be a $C^2$ injective curve, then there exists $\tilde n_0$ such that
\[ \forall x \: \vert \: {\rm d}(x,\Cc) <1/\tilde n_0,\: \exists ! y \in \Cc \: \vert \: {\rm d}(x,y)={\rm d}(x,\Cc).\]
\end{lemma}

\begin{proof}
Let $s\mapsto \Cc(s)=M$ be the arclength parametrization. Then, we recall the standard definition:
\begin{itemize}
\item the tangent vector and the curvature
\[ \overrightarrow{\t}(s)=\df{M}{s} = \Cc'(s), \quad \r(s)= \Bigl\| \df{\overrightarrow{\t}(s)}{s}\Bigl\|=\| \Cc''(s) \| ;\]
\item the main normal vector and the binormal vector (if $\r(s)\neq 0$)
\[ \overrightarrow{\n}(s)=\frac1{\r(s)} \df{\overrightarrow{\t}(s)}{s}= \frac{\Cc''(s)}{\| \Cc''(s) \|} ,\quad \overrightarrow{\b}(s) = \overrightarrow{\t}(s)\times \overrightarrow{\n}(s).\]
\end{itemize}
As $\Cc$ is assumed to be $C^2$, there exists $\r_M>0$ such that $|\r(s) |\leq \r_M$ on $\Cc$.

The existence of $y\in \Cc$ such that ${\rm d}(x,y)={\rm d}(x,\Cc)$ is obvious because the map $z\mapsto {\rm d}(x,z)$ is continuous on the compact $\Cc$. Let us assume that the conclusion of the lemma is false, i.e. that the infimum is reached twice:
\[\forall n\in \N, \:  \exists x_n \: \vert \: {\rm d}(x_n,\Cc) <1/n,\: \exists y_n=\Cc(s_{1,n}), z_n=\Cc(s_{2,n}) \]
such that
\[ s_{1,n}<s_{2,n} \text{ and } {\rm d}(x_n,y_n)={\rm d}(x_n,\Cc)= {\rm d}(x_n,z_n).\]
Extracting a subsequence, we have by compactness that $x_n, y_n,z_n \to x \in \Cc$. As the curve is assumed to be injective (so without cross-point), we infer that $s_{1,n}-s_{2,n}\to 0$. So, by continuity of $\Cc$ in $x$, we note easily that there exists $N$ such that
\begin{equation}\label{near}
{\rm d}(x_N,\Cc(s)) \leq 1/(2\r_M),\: \forall s \in [s_{1,N}, s_{2,N}]. 
\end{equation}

Next, we introduce the following function
\[f: \: s\mapsto {\rm d}(x_N, \Cc(s))^2= \| \overrightarrow{x_N M(s)} \|^2,\]
and we differentiate twice to get
\[ f'(s)= 2 \overrightarrow{x_N M(s)} \cdot \overrightarrow{\t}(s),\quad f''(s)=2\Bigl(1+\r(s) \overrightarrow{x_N M(s)} \cdot \overrightarrow{\n}(s)\Bigl).\]
The previous computation holds even if $\r(s)=0$.

By assumption, $f$ is minimal for $s_{1,N}$ and $s_{2,N}$, so
\[ f'(s_{1,N})=0=f'(s_{2,N}),\]
but \eqref{near} implies that
\[ f''(s) \geq 1,\: \forall s \in [s_{1,N}, s_{2,N}] \]
which is impossible. This contradiction allows us to end this proof.
\end{proof}

\begin{remark}\label{rem perp}
In the previous proof, we infer from $f'(s_{1,N})=0$ that $\overrightarrow{x_N M(s_{1,N})} \cdot \overrightarrow{\t}(s_{1,N})=0$ when $M(s_{1,N})$ is the closest point of $\Cc$ to $x_N$.
\end{remark}

Now, we link the differentiability of $g\: : \: x\mapsto {\rm d}(x,\Cc)$ and the previous lemma.

\begin{lemma} \label{d diff}
Let $\Cc$ be a $C^2$ injective curve and $\tilde n_0$ from Lemma \ref{unique-min}. Then for all $x\in \R^3\setminus\Cc$ such that ${\rm d}(x,\Cc) <1/\tilde n_0$, $g$ is differentiable in $x$ and
\[ \na g(x) =  \frac{x-y}{\|x-y\|},\]
where $y$ is the unique point of $\Cc$ such that $\|x-y\|={\rm d}(x,\Cc)$.
\end{lemma}

\begin{proof} This proof can be found in the exercise book \cite{gonnord} and we copy it for a sake of completeness.
We fix $x\in \R^3\setminus\Cc$ such that ${\rm d}(x,\Cc) <1/\tilde n_0$, then Lemma \ref{unique-min} states that there exists a unique $y\in \Cc$ verifying $\|x-y\|={\rm d}(x,\Cc)$.

The first step consists to show the following property: for all $h\in \R^3$ small enough, we set $y_h$ by the unique point of $\Cc$ such that $\|x+h-y_h\|={\rm d}(x+h,\Cc)$, then
\begin{equation}\label{limit y_h}
\|y_h-y \| \to 0 \text{ as } h\to 0.
\end{equation}
Let us assume that \eqref{limit y_h} is false, then there exist a sequence $h_n \to 0$ and $\d>0$ such that $\|y_{h_n}-y \| \geq \d$ for all $n$. We introduce the following compact set
\[ G:= \{ z\in \Cc \: \vert \: \|z-y\|\geq \d\}\]
and the function
\[h\: : \: z\in G \mapsto \|x-z\|^2.\]
This function is continuous and reaches his minimum. As $y\notin G$, the uniqueness part of Lemma \ref{unique-min} allows us to check that this minimum is strictly greater than $\|x-y\|^2$.
Hence, we have
\[ {\rm d}(x,G)> {\rm d}(x,\Cc).\]
However, we use the fact that $y_{h_n}$ belongs to $G$ to write
\[g(x+h_n)= {\rm d}(x+h_n,\Cc)=\|x + h_n - y_{h_n}\|={\rm d}(x+h_n,G)\]
implying that
\[ g(x+h_n) \to {\rm d}(x,G)> g(x)\]
which is impossible by the continuity of $g$. The property \eqref{limit y_h} follows from this contradiction.

\medskip
Let us prove in the second step that \eqref{limit y_h} implies that $g$ is differentiable. On the one hand, we have that
\begin{equation}\label{g sup}
 g^2(x+h) \leq \|x+h-y\|^2= g^2(x) + 2 (x-y)\cdot h + \|h\|^2.
\end{equation}
On the other hand, we compute
\begin{eqnarray*}
g^2(x+h)= \|x+h-y_h \|^2 &=& \| x-y_h\|^2 + 2 (x-y_h)\cdot h + \| h \|^2\\
&\geq& \| x-y \|^2 + 2 (x-y)\cdot h + 2 (y-y_h)\cdot h + \| h \|^2.
\end{eqnarray*}
Hence, \eqref{limit y_h} implies that
\begin{equation}\label{g inf}
g^2(x+h) \geq g^2(x) + 2 (x-y)\cdot h + o(h).
\end{equation}
Putting together \eqref{g inf}-\eqref{g sup}, we conclude that $g^2$ is differentiable and that
\[ \nabla (g^2(x))=2g(x)\na g(x) = 2 (x-y)\]
which ends the proof.
\end{proof}

From the value of $\na g$, we see that the uniqueness part of Lemma \ref{unique-min} is mandatory to obtain the differentiability of $g$ (or of $g^2$).

Finally, we can establish the estimates of the cutoff function by proving Proposition \ref{cutoff est}.

\medskip

Thanks to the definition and Lemma \ref{d diff}, we compute for $n>2 \tilde n_0$
\[ \na \y_n(x)= n \chi' \Bigl( n {\rm d}(x,\Cc) \Bigl) \na {\rm d}(x,\Cc)= n \chi' \Bigl( n {\rm d}(x,\Cc) \Bigl)  \frac{x-y}{\|x-y\|},\]
where $y$ is the unique point of $\Cc$ such that $\|x-y\|={\rm d}(x,\Cc)$. Indeed, $\chi' \Bigl( n {\rm d}(x,\Cc) \Bigl)\neq 0$ iff $1/n\leq {\rm d}(x,\Cc)\leq 2/n$, so $x\in \R^3\setminus \Cc$ and  ${\rm d}(x,\Cc)< 1/\tilde n_0$ which allows us to apply the previous lemmas.

Hence, we obtain directly the first point:
\[\| \na \y_n \|_{L^\infty(\R^3)}\leq n \| \chi' \|_{L^\infty(\R)}.\]

\medskip

Concerning the support, we use Remark \ref{rem perp} to note that for all $x\in \supp (1 -\y_n)$, then the unique $y\in \Cc$ such that ${\rm d}(x,y)={\rm d}(x,\Cc)$ verifies $\overrightarrow{xy}\cdot \overrightarrow{\t}=0$. If there is an interval where $\r(s)=0$, then it implies that $\Cc$ on this part is a segment. Then, it is obvious that the volume such that $ {\rm d}(x,\Cc)\leq 2/n$ is $\mathcal{O}(1/n^2)$ on this section. So, without loss of generality, we assume that $\r(s)\neq 0$ on $\Cc$, and the previous remark implies that
\[ \supp{(1-\y_n)} = \{ \Cc(s) + r\cos\f \overrightarrow{\n}(s)+r\sin\f \overrightarrow{\b}(s) \: \vert \: (s,r,\f)\in [0,L]\times [0,2/n]\times [0,2\pi)    \} \bigcup_{i=1,2} O_i  \]
and the application $\P$: $(s,r,\f)\mapsto \Cc(s) + r\cos\f \overrightarrow{\n}+r\sin\f \overrightarrow{\b}$ is a diffeomorphism invertible where $\p^{-1}(M)$ gives the parameter of $\Cc(s)$ where ${\rm d}(M,\Cc)$ is reached, the distance ${\rm d}(M,\Cc)$, and the angle between $\overrightarrow{\Cc(s)M}$ and $\overrightarrow{\n}$. In the expression of $\supp (1-\y_n)$, $O_i$ denotes the half ball on each endpoints, which obviously verify ${\rm mes}\: O_i=\mathcal{O}(1/n^3)$.

 Using the general relation
\[ \df{\overrightarrow{\b}}{s}=\th(s) \overrightarrow{\n}(s), \quad \df{\overrightarrow{\n}}{s}=-\r(s) \overrightarrow{\t}(s)-\th(s) \overrightarrow{\b}(s),\]
where $\theta$ is a smooth function called the torsion, then we can compute $D\P$ in the orthonormal base $(\overrightarrow{\t},\overrightarrow{\n},\overrightarrow{\b})$:
\[D\P(r,s,f) = \begin{pmatrix} 1- r \r\cos \f&0&0 \\ r \th \sin \f  & \cos \f & -r \sin \f \\ -r \th \cos \f & \sin \f & r \cos\f \end{pmatrix}\]
which implies that $|\det D\P(r,s,f)| \leq  r/2$ if $r$ is chosen small enough (i.e. $r\leq 1/(2 \r_M)$).
Finally, we can compute for $n$ big enough:
\begin{eqnarray*}
 {\rm mes}\Bigl( \supp{(1- \y_n)}\Bigl) &=& \int_{\supp{(1-\y_n)}} 1 \, dxdydz= \int_0^L \int_{0}^{2/n} \int_0^{2\pi} |\det D\P(r,s,f)| \, d\f dr ds +\mathcal{O}(1/n^3)\\
 &\leq& \int_0^L \int_{0}^{2/n} \int_0^{2\pi} r/2 \, d\f dr ds +\mathcal{O}(1/n^3) = \frac{2\pi L}{n^2}+\mathcal{O}(1/n^3) 
\end{eqnarray*}
which ends the proof of Proposition \ref{cutoff est}.

\bigskip

\noindent
 {\bf Acknowledgements.} I am partially supported by the Agence Nationale de la Recherche, Project MathOc\'ean, grant ANR-08-BLAN-0301-01 and by the Project ``Instabilities in Hydrodynamics'' funded by Paris city hall (program ``Emergences'') and the Fondation Sciences Math\'ematiques de Paris.

\end{document}